\documentclass{article}

\usepackage[
  lang=american,
]{ems-journal}
\usepackage{mathtools}
\usepackage{mathrsfs}
\DeclareFontFamily{U}{rsfs}{\skewchar\font127 }
\DeclareFontShape{U}{rsfs}{m}{n}{%
  <-6.5> rsfs5 <6.5-8> rsfs7 <8-> rsfs10}{}
\usepackage{microtype}
\usepackage{comment}
\usepackage{needspace}
\usepackage{tikz}
\usetikzlibrary{arrows.meta,calc,decorations.pathreplacing}

\definecolor{Bfill}{RGB}{100,170,222}
\definecolor{Bedge}{RGB}{49,111,157}
\definecolor{ABfill}{RGB}{239,156,66}
\definecolor{ABedge}{RGB}{190,104,28}

\newcommand{\dd}{\mathop{}\!\mathrm{d}}
\newcommand{\one}{\mathbf 1}
\DeclareMathOperator{\Id}{Id}
\DeclareMathOperator{\dist}{dist}
\DeclareMathOperator*{\esssup}{ess\,sup}
\newcommand{\R}{\mathbb{R}}

\theoremstyle{plain}
\newtheorem{theorem}{Theorem}[section]
\newtheorem{proposition}[theorem]{Proposition}
\newtheorem{lemma}[theorem]{Lemma}
\newtheorem{corollary}[theorem]{Corollary}

\theoremstyle{definition}
\newtheorem{remark}[theorem]{Remark}

\numberwithin{equation}{section}
\makeatletter
\ifdefined\HyOrg@maketitle
  \let\ems@titlewithdedication\HyOrg@maketitle
\else
  \let\ems@titlewithdedication\maketitle
\fi
\xpatchcmd{\ems@titlewithdedication}
  {\ifx\ems@dedication\@undefined\else
    \addvspace\ems@grid
    {\ems@titleblockalignment\itshape\ems@dedication\par}%
   \fi}
  {}
  {}
  {\PackageError{Final-EMS}{Could not remove the original dedication block}%
    {Use the ems-journal.sty supplied with this project.}}
\xpatchcmd{\ems@titlewithdedication}
  {\endgroup\vskip5mm}
  {\ifx\ems@dedication\@undefined\else
     \addvspace\ems@grid
     {\itshape\ems@dedication\par}%
   \fi
   \endgroup\vskip5mm}
  {}
  {\PackageError{Final-EMS}{Could not place the dedication below the authors}%
    {Use the ems-journal.sty supplied with this project.}}
\ifdefined\HyOrg@maketitle
  \let\HyOrg@maketitle\ems@titlewithdedication
\else
  \let\maketitle\ems@titlewithdedication
\fi
\makeatother

\begin{document}

\title{On $L^p$ norms of maximal operators}
\titlemark{On $L^p$ norms of maximal operators}

\emsauthor{1}{
  \givenname{Valentina}
  \surname{Ciccone}}{V.~Ciccone}
\Emsaffil{1}{
  \organisation{Institute of Mathematics, Polish Academy of Sciences}
  \address{{\'S}niadeckich 8}
  \zip{00-656}
  \city{Warszawa}
  \country{Poland}
  \affemail{vciccone@impan.pl}}

\emsauthor{2}{
  \givenname{Mahdi}
  \surname{Hormozi}}{M.~Hormozi}
\Emsaffil{2}{
  \organisation{Beijing Institute of Mathematical Sciences and Applications}
  \zip{101408}
  \city{Beijing}
  \country{China}
  \affemail{hormozi@bimsa.cn}}

\emsauthor{3}{
  \givenname{Jos{\'e}}
  \surname{Madrid}}{J.~Madrid}
\Emsaffil{3}{
  \department{Department of Mathematics}
  \organisation{Virginia Tech}
  \address{225 Stanger Street}
  \zip{24061-1026}
  \city{Blacksburg, VA}
  \country{USA}
  \affemail{josemadrid@vt.edu}}

\emsauthor{4}{
  \givenname{Jakub}
  \surname{Niksiński}}{J.~Niksiński}
\Emsaffil{4}{
  \department{Department of Mathematics}
  \organisation{Rutgers University}
  \address{Hill Center -- Busch Campus, 110 Frelinghuysen Road}
  \zip{08854-8019}
  \city{Piscataway, NJ}
  \country{USA}
  \affemail{jakub.niksinski.math@gmail.com}}

\emsauthor{5}{
  \givenname{B{\l}a{\.z}ej}
  \surname{Wr{\'o}bel}}{B.~Wr{\'o}bel}
\Emsaffil{5}{
  \organisation{1}{Institute of Mathematics, Polish Academy of Sciences}
  \address{1}{{\'S}niadeckich 8}
  \zip{1}{00-656}
  \city{1}{Warszawa}
  \country{1}{Poland}
  \department{2}{Institute of Mathematics}
  \organisation{2}{University of Wroc{\l}aw}
  \address{2}{Plac Grunwaldzki 2}
  \zip{2}{50-384}
  \city{2}{Wroc{\l}aw}
  \country{2}{Poland}
  \affemail{2}{bwrobel@impan.pl}}

\makeatletter
\preto\ems@authlist{\@namedef{@sep4}{,\\}}
\preto\ems@affillist{\Needspace{24\baselineskip}}
\makeatother

\dedication{Dedicated to Leonard Gross}
\classification[47D07, 60G42]{42B25}

\hypersetup{
  pdftitle={On L\textasciicircum p norms of maximal operators},
  pdfauthor={Valentina Ciccone, Mahdi Hormozi, José Madrid, Jakub Niksiński, Błażej Wróbel},
  pdfsubject={L\textasciicircum p norms of maximal functions associated to semigroups and convolution operators},
  pdfkeywords={maximal operator, symmetric Markov semigroup, heat semigroup,
    discrete Gaussian, Ornstein--Uhlenbeck semigroup, Stein's maximal
    inequality, centered Hardy--Littlewood maximal operator, sharp constant,
    near-extremizer}
}

\begin{abstract}
We study the exact $L^p$ norms for a number of maximal operators, which are prominent in harmonic analysis, semigroup theory, and probability. We prove that a strongly continuous positive symmetric sub-Markovian semigroup $(T_t)_{t\ge0}$ on a nonatomic $\sigma$-finite measure space $(X,\mathcal F,\mu)$ has maximal $L^p$ norm $p'=p/(p-1)$ for every $1<p<\infty$ whenever $f\mapsto T_{t_0}(\one_Ef)$ is compact on $L^2(\mu)$ for some $t_0>0$ and some measurable set $E$ of positive measure. The lower bound follows by approximating finite dyadic conditional expectations at suitable semigroup times; the upper bound is Stein's maximal inequality in its sub-Markovian form.

Important examples of such maximal operators we discuss include maximal functions of various semigroups (heat, Poisson, Ornstein--Uhlenbeck, Schr\"odinger, Dunkl, Laguerre, Jacobi) as well as the centered Hardy-Littlewood maximal operator on $\mathbb{R}.$
In particular, the heat and Poisson maximal operators on every
complete connected Riemannian manifold of positive dimension without
boundary have exact $L^p$ norm $p'$, with no curvature or stochastic
completeness assumptions.

Combining our approach with previous work of the first and last authors we also prove that $p'$ is the limit, as $d\to \infty,$ of the $L^p$ norms $B(p,d)$ of the centered Hardy-Littlewood maximal operators on $\R^d.$ On the other hand, we show that the norm of this maximal operator is strictly larger than $p',$ in dimensions $d\ge2.$

\end{abstract}

\makeatletter
\let\ps@titlepage\ps@empty
\makeatother
\maketitle

\section{Introduction}
\label{sec:introduction}
A central theme of various fields of analysis, including harmonic analysis, geometric analysis, semigroup theory, and probability theory, is the study of
inequalities for operators. Prominent examples of these operators are the maximal functions. Three recurring questions concern
the range of exponents for which strong- or weak-type estimates hold,
the dependence of the bounds on dimension and other parameters, and
the determination of optimal constants. 

One of the most important maximal functions in probability theory is the one associated to an increasing filtration. Consider a probability space
$(\Omega,\mathcal G,\mathbb P)$, let
$\mathcal G_0\subseteq\cdots\subseteq\mathcal G_N\subseteq \cdots \subseteq \mathcal G$ be an increasing filtration of $\mathcal G$, and define the conditional expectation 
\[\mathcal E_n f= \mathbb E(f\mid \mathcal G_n).\] Then $(\mathcal E_n f)_{n\ge0}$ is a martingale and Doob's maximal inequality implies that, for all $p\in (1,\infty)$,
\[
\|\sup_{n\ge 0}\left|\mathcal E_n f\right|\|_{L^p}
\le
p'\|f\|_{L^p}.
\]
Furthermore,  the constant $p'$ is optimal; see
\cite{Burkholder1986,Burkholder1991} for general martingales
and \cite{Wang1991} for dyadic martingales.

Connections to martingales and Doob's maximal inequality have been extensively exploited in estimating maximal functions in other fields of analysis, including harmonic analysis and semigroup theory. A notable example is Rota's construction, which implies that the norm of a maximal function corresponding to a strongly continuous positive symmetric Markov semigroup does not exceed $p'.$  The main purpose of our work is to develop an easy-to-check local compactness condition for semigroups that implies that the norm is in fact equal to $p'.$  The mechanism behind our condition is the approximation of finite dyadic conditional expectations by the semigroup at suitably chosen times.

Our local compactness result is the following theorem.
\begin{theorem}\label{thm:local-compactness}
Let $(T_t)_{t\ge0}$ be a strongly continuous positive symmetric sub-Markovian
semigroup on a nonatomic $\sigma$-finite measure space $(X,\mathcal F,\mu)$.
If $f\mapsto T_{t_0}(\one_Ef)$ is compact on $L^2(\mu)$ for some $t_0>0$
and $E\in\mathcal F$ of positive measure, then
\[
 \|T^*\|_{L^p(\mu)\to L^p(\mu)}=p'
 \qquad(1<p<\infty).
\]
\end{theorem}
\noindent In other words, Theorem \ref{thm:local-compactness} asserts that the norm equals $p'$ whenever the spatially localized semigroup is a compact operator at some time $t_0$. 

This result implies that the norms of the maximal functions corresponding to the classical heat and Poisson, Jacobi, Hermite, Laguerre, Ornstein--Uhlenbeck, and Dunkl semigroups, as well as their subordinates associated with unbounded Bernstein functions $\Phi$ satisfying $\Phi(0)=0$, are all equal to $p'$; see Corollary \ref{cor: semigroup examples} and Corollary \ref{cor:bernstein-maximal-norms}.
As an application to geometric analysis, our theorem gives
the exact $L^p$ norm $p'$ for the heat and Poisson maximal operators
on every complete connected Riemannian manifold of positive
dimension without boundary, for $1<p<\infty$, without curvature
or stochastic completeness assumptions.

We now discuss corollaries of Theorem \ref{thm:local-compactness} for the values of the norms of the centered Hardy-Littlewood maximal function of the balls $\mathscr B_*^d$ in $\R^d.$

\begin{corollary}\label{cor:ball-lower-bounds}
For every $1<p<\infty$ and every $d\geq1$,
\begin{equation}\label{eq:intro-continuous-balls}
\begin{aligned}
\|\mathscr B_*^d\|_{L^p(\mathbb R^d)\to L^p(\mathbb R^d)}\ge p'.
\end{aligned}
\end{equation}
\end{corollary}
\begin{corollary}
\label{cor:HL-norm-R}
For every $1<p<\infty$ we have
\[
 \|\mathscr B_*^1\|_{L^p(\mathbb R)\to L^p(\mathbb R)}=p'.
\]
\end{corollary}
Using the work of the first and last authors \cite{CicconeWrobel}, we can also compute the limit of the norms.
\begin{corollary}\label{cor:ball-limit}
For every $1<p<\infty$ we have
\begin{equation}\label{eq:intro-continuous-balls}
\begin{aligned}
 \lim_{d\to \infty}\|\mathscr B_*^d\|_{L^p(\mathbb R^d)\to L^p(\mathbb R^d)}=p'.
\end{aligned}
\end{equation}
\end{corollary}

It turns out that, for $d\ge2$, the lower bound is strict.

\begin{proposition}\label{prop:HL-norm-Rd}
For every integer $d\ge2$ and every $1<p<\infty$,
\[
  \|\mathscr B_*^d\|_{L^p(\mathbb R^d)\to L^p(\mathbb R^d)}>p'.
\]
\end{proposition}
\noindent Proposition \ref{prop:HL-norm-Rd} is our only result whose proof was developed with AI tools. We present it in Appendix~\ref{sec:strict-HL-lower-bound}.

\subsection{Historical background}

There is a substantial literature devoted to the determination of exact
$L^p$ norms of classical operators in harmonic analysis. Perhaps the
best known example is Pichorides' \cite{Pichorides1972} theorem for the Hilbert transform,
which gives
\[
 \|H\|_{L^p(\mathbb R)\to L^p(\mathbb R)}
 =
 \cot\frac{\pi}{2p^*},
 \qquad
 p^*=\max\{p,p'\}.
\]
The same constant is the exact norm of each of the Euclidean Riesz
transforms; see, for instance, the work of Iwaniec and Martin \cite{IwaniecMartin1996}. Exact
norms are also known for a number of other closely related singular
integrals.

However, for maximal operators the situation is rather different. Exact strong
$L^p$ norms are known only in relatively few cases, and even for some
of the most classical maximal functions the corresponding problem has
remained open for quite some time.

On the other hand, that $p'$ is the upper bound for maximal functions of strongly continuous positive symmetric Markov semigroups is nowadays a classical result. A fundamental connection
between martingales and positive operators was developed by Rota
\cite{Rota1962}. Then, together with Doob's maximal inequality, it became an important ingredient in Stein's theory of
symmetric diffusion semigroups \cite{Stein1970}. In particular, Stein's
maximal theorem gives a general framework in which martingale methods
can be used to study maximal operators associated with semigroups. His maximal theorem implies that
for any maximal function $T_* f(x)=\sup_{t>0}|T_t f(x)|$ corresponding to a strongly continuous
positive symmetric Markov semigroup $\{T_t\}_{t>0}$ on a $\sigma$-finite measure space
$(X,\mathcal F,\mu)$ we have
\[
 \|T^*f\|_{L^p(\mu)}\leq p'\|f\|_{L^p(\mu)}.
\]
The maximal theorem is developed in
\cite[Chapter~III]{Stein1970}; the constant $p'$ follows from the
martingale argument in \cite[Chapter~IV, Section~4]{Stein1970}.
However, whether $p'$ is the norm of $T^*$ depends on the semigroup itself.
Sharpness uniformly over a class of semigroups does not imply
equality for a fixed semigroup. Theorem~\ref{thm:local-compactness}
gives a sufficient condition for equality by approximating finite
dyadic conditional expectations at suitable semigroup times.

Further developments include Cowling's semigroup maximal theorems
\cite[Section~3]{Cowling1983} and the sectorial maximal estimates and
pointwise convergence results of Carbonaro and Dragi\v{c}evi\'c
\cite[Theorem~12 and Corollary~13]{CarbonaroDragicevic2017}.

Maximal operators associated with semigroups have been studied in a
variety of classical settings. The basic examples are the heat and
Poisson semigroups, which already played an important role in the
development of Littlewood--Paley theory. Stein's theory of symmetric
diffusion semigroups \cite{Stein1970} placed these examples in a common
abstract framework and established maximal inequalities for a large
class of semigroups. Other important examples arise from the
Ornstein--Uhlenbeck semigroup and from semigroups associated with
classical orthogonal expansions, including Hermite, Laguerre, Jacobi
and Bessel expansions. The corresponding maximal operators have been
studied from several points of view, with particular attention to
strong $L^p$ boundedness, weak type $(1,1)$ estimates, regularity on Sobolev spaces and almost
everywhere convergence. For the Ornstein--Uhlenbeck setting we refer,
among others, to
\cite{CasarinoCiattiSjogren2022,MenarguezPerezSoria2000,PerezSoria2000,Sjogren1983}.  Maximal operators associated with Laguerre,
Jacobi and Bessel heat and Poisson semigroups were considered in
\cite{Carneiro2019,CarneiroSvaiter2013,NowakSjogren2012,NowakSjogren2013,NowakSjogrenSzarek2017,NowakSzarek2012}; see also the references therein. In contrast to
boundedness and endpoint questions, considerably less seems to be known
about the exact strong $L^p$ norms of these operators.

Having discussed $L^p$ norms of semigroups, we now focus on norms of Hardy-Littlewood maximal operators. As for the uncentered
Hardy--Littlewood maximal operator on $\mathbb R$, its exact $L^p$ norm was determined by Grafakos and Montgomery-Smith \cite{GrafakosMontgomerySmith1997}.

The problem of determining the exact norm is quite delicate for the classical centered Hardy--Littlewood maximal operator
\[
 \mathscr B_{*}^1 f(x)=\sup_{t>0}\frac{1}{2t}
       \int_{x-t}^{x+t}|f(y)|\,dy
\]
on $\mathbb R$. 
Melas \cite{Melas2003} determined the exact weak type $(1,1)$
bound for $\mathscr B_*^1$.
By contrast, the problem of determining the sharp strong
$L^p$ norm of the centered Hardy--Littlewood maximal operator,
$1<p<\infty$, remained open even on the real line; see, e.g., the survey of Deleaval, Gu\'edon and Maurey
\cite[p.~3]{DeleavalGuedonMaurey2018}.

The elementary comparison with the one-sided maximal
operators gives
\[
 \|\mathscr B_*^1\|_{L^p(\mathbb R)\to L^p(\mathbb R)}\leq p'.
\]
The best bound was investigated numerically by Dror, Ganguli and Strichartz \cite{DrorGanguliStrichartz1996}.  They observed that
$|x|^{-1/p}$ is a formal eigenfunction of $\mathscr B_*^1$ and conjectured that the
corresponding eigenvalue gives the exact $L^p$ norm. Equivalently, their
candidate for the norm of $\mathscr B_*^1$ was
\[
 c_p
 =
 p'\sup_{\tau>1}
 \frac{(\tau+1)^{1-1/p}+(\tau-1)^{1-1/p}}{2\tau}.
\]
They proved that $c_p$ is a lower bound and supported the conjecture by
numerical evidence. Subsequently, Grafakos, Montgomery-Smith and Motrunich
\cite{GrafakosMontgomerySmithMotrunich1999}  proved that
$c_p$ is the optimal constant within the class of peak-shaped
functions, namely, nonnegative functions that are convex on each side of one point. Note that for $\tau>1$, strict concavity of $u\mapsto u^{1-1/p}$ gives
\[
 \frac{(\tau+1)^{1-1/p}+(\tau-1)^{1-1/p}}{2\tau}
 <\tau^{-1/p}<1.
\]

The quotient extends continuously to $\tau=1$, where its value is
$2^{-1/p}$, and tends to zero as $\tau\to\infty$.  Therefore its supremum is
strictly smaller than one.  Hence $c_p<p'$, so, by Corollary \ref{cor:HL-norm-R}, the conjectured
value is not the operator norm on all of $L^p(\mathbb R).$ 

Similarly, not much was known about the exact norm of the Hardy-Littlewood maximal operator on $\R^d$
\[
 \mathscr B_{*}^{d} f(x)=\sup_{t>0}\frac{1}{|B_t|}
       \int_{B_t}|f(x-y)|\,dy.
\]
From the work of the first and last authors \cite{CicconeWrobel}, it follows that 
\[\|\mathscr B_*^d\|_{L^p(\mathbb R^d)\to L^p(\mathbb R^d)}\ge \frac25 p'\] and that the limit
\[
\lambda(p):=\lim_{d\to \infty}\|\mathscr B_*^d\|_{L^p(\mathbb R^d)\to L^p(\mathbb R^d)}
\]
exists. The same work identifies this limit with the high-dimensional
limit of the Gaussian maximal norms. Using the exact heat maximal norm, we obtain the general lower bound $p'$ and the identity $\lambda(p)=p'$.

We also mention that sharp constants have been studied extensively for martingale and
dyadic maximal operators. For the dyadic maximal operator, and more
generally for tree-like maximal operators, Melas
\cite{Melas2005} computed the corresponding Bellman functions. The
structure of extremal and approximately extremal sequences was studied
further in \cite{Melas2019,Nikolidakis2017}. More recently, Kosz \cite{Kosz2024} studied sharp constants in inequalities admitting
the Calder\'on transference principle.


\subsection*{Acknowledgments}
V.C. and B.W. were supported by the National Science Centre, Poland, grant Sonata Bis 2022/46/E/ST1/00036.
J.M. was partially supported by the Simons Foundation Grant \#453576. J.N. was supported by the NSF CAREER grant DMS-2236493.\\

\subsection*{AI Statement}
The mathematical ideas and results in the main body of this paper
are due entirely to the authors, who developed the arguments and
wrote the proofs. AI assistance in the main body was limited to
editorial work. In contrast, the proof of
Proposition~\ref{prop:HL-norm-Rd} was obtained with crucial assistance
from GPT-6 Astra.  To distinguish this contribution from the authors'
arguments, we present that proof separately in
Appendix~\ref{sec:strict-HL-lower-bound}.
The authors take responsibility for the mathematical content of
the entire paper.

\section{Preliminaries}\label{sec:prelim}
\subsection{Stein's maximal inequality}
Let $(X,\mathcal F,\mu)$ be a $\sigma$-finite measure space.
A \emph{strongly continuous positive symmetric sub-Markovian
semigroup} is a family $(T_t)_{t\geq0}$ of operators acting on
every $L^q(\mu)$, $1\leq q\leq\infty$, whose realizations agree
on intersections of these spaces, such that
\[
 T_0=\Id,\qquad T_sT_t=T_{s+t},\qquad
 0\leq T_t\one\leq\one,
\]
each $T_t$ is positive, contractive on every $L^q(\mu)$,
and self-adjoint on $L^2(\mu)$, and
\[
 \|T_tf-f\|_{L^2(\mu)}\longrightarrow0
 \qquad(t\downarrow0),\qquad f\in L^2(\mu).
\]
The $L^\infty$ realization is the Banach adjoint of the
$L^1$ realization. The inequality $T_t\one\leq\one$ is
understood in $L^\infty(\mu)$.
If $T_t\one=\one$ for every $t\geq0$, we call the semigroup
\emph{Markov}.

It can be easily verified that the semigroup is strongly continuous on every
$L^q$, $1<q<\infty$. Namely,
for every $1<q<\infty$ and every $g\in L^q(\mu)$,
\begin{align}\label{eq:Lp-strong-continuity}
 \|T_tg-g\|_{L^q(\mu)}\longrightarrow0\qquad(t\downarrow0).
\end{align}
Moreover, $t\mapsto T_tg$ is continuous from $[0,\infty)$ to
$L^q(\mu)$.

Throughout the paper, for $1<p<\infty$ we let \[p'=p/(p-1).\]
We say that a linear map $J: L^p(Y) \to L^p(X)$ is a lattice embedding if it is injective and
\[
J(|f|)=|J(f)|
\]
for every $f\in L^p(Y)$. Such a map preserves finite maxima of nonnegative functions.
For $f\in L^p(\mu)$, we define
\[
 T^*f=\esssup_{t>0}T_t|f|.
\]
Positivity gives
\[
 \esssup_{t>0}|T_tf|\leq T^*f.
\]
Applied to $|f|$, the left-hand side equals $T^*f$. Hence $T^*$ and
$f\mapsto\esssup_{t>0}|T_tf|$ have the same $L^p$-operator norm.


The classical argument outlined in \cite[Chapter~IV, Section~4]{Stein1970} applies Rota's theorem \cite[Theorem~2(a), p.~101]{Rota1962} in its $\sigma$-finite formulation \cite[Chapter~IV, Section~4, Theorem~9, pp.~106--110]{Stein1970}. Combined with a $\sigma$-finite reformulation of Doob's estimates \cite[Chapter~VII, Section~3, Theorem~3.4, p.~317]{Doob1953}, this application yields the maximal theorem for positive symmetric Markov semigroups. To extend this result to sub-Markovian semigroups, one may instead use Starr's extension of Rota's theorem to positive contractions on $\sigma$-finite measure spaces \cite{Starr1966}. Its quantitative form,
with the sharp Doob constant $p'=p/(p-1)$, follows from \cite[Theorem~4.7 and Estimate~A, p.~489; p.~490]{Bradley1991}. For the convenience of the reader, we provide a sketch of the proof below.

\begin{theorem}\label{thm:stein-maximal}
Let $1<p<\infty$, and let $(T_t)_{t\geq0}$ be a strongly continuous
positive symmetric sub-Markovian semigroup on a $\sigma$-finite measure space
$(X,\mathcal F,\mu)$.  Then, 
for every $f\in L^p(\mu)$,
\[
 \|T^*f\|_{L^p(\mu)}\leq p'\|f\|_{L^p(\mu)}.
\]
Equivalently,
\[
 \left\|\esssup_{t>0}|T_tf|\right\|_{L^p(\mu)}
 \leq p'\|f\|_{L^p(\mu)}.
\]
\end{theorem}
\begin{proof}[Proof (sketch)]
Assume that $P$ is a positive contraction on $L^1$ and $L^{\infty}$ which is self-adjoint on $L^2$ and satisfies the sub-Markov inequality $P \one \leq \one.$ Then, \cite{Starr1966} implies that
\[
    \left\|\sup_{n\geq0}|P^{2n} f|\right\|_{L^p}
    \leq p' \|f\|_{L^p};
\]
see \cite[Estimate~A, p.~489; see also p.~490]{Bradley1991} for the exact constant $p'.$
Applying this estimate to $P=T_{h/2},$ $h>0,$ and using the semigroup property $P^{2n}=T_{nh},$ we obtain
\[
    \left\|\sup_{n\geq0}|T_{nh}f|\right\|_{L^p}
    \leq p'\|f\|_{L^p}.
\]
Now, taking $h=2^{-k}$, letting $k\to\infty,$ and using the monotone convergence theorem, we obtain the maximal inequality for the supremum restricted to positive dyadic numbers. Finally, approximating arbitrary $t>0$ by dyadic numbers and
using the strong $L^p$-continuity of $(T_t)_{t\geq0}$ yields
\[
    \left\|\esssup_{t>0}|T_tf|\right\|_{L^p}
    \leq p'\|f\|_{L^p},
\]
as desired.
\end{proof}

The maximal theorem above readily implies the upper bound in Theorem \ref{thm:local-compactness}. The proof of the lower
bound uses positivity, contractivity, symmetry, and strong
continuity, but does not require preservation of constants.

\subsection{Finite dyadic maximal constants}

To establish the lower bound in Theorem \ref{thm:local-compactness}, we will make essential use of the finite dyadic maximal constants defined below and the fact that they increase to $p'$.

For each integer $N\geq0$, let
\[
 \Omega_N=\{-1,1\}^N
\]
with normalized counting measure.  For $0\leq n\leq N$, define
\[
 E_nF(\omega)
 =2^{n-N}\!\!\!\sum_{\sigma_{n+1},\ldots,\sigma_N\in\{-1,1\}}
 F(\omega_1,\ldots,\omega_n,
   \sigma_{n+1},\ldots,\sigma_N).
\]
When $n=N$, the sum consists of the single term $F(\omega)$.  Put
\begin{equation}
 D_{p,N}
 =\sup_{\substack{F\geq0\\F\neq0}}
 \frac{\left\|\max_{0\leq n\leq N}E_nF\right\|_{L^p(\Omega_N)}}
      {\|F\|_{L^p(\Omega_N)}}.
 \label{eq:finite-dyadic-definition}
\end{equation}
Since $|E_nF|\leq E_n|F|$, this is also the operator norm of the map
\[
 F\longmapsto\max_{0\leq n\leq N}|E_nF|.
\]

\begin{proposition}
\label{prop:finite-dyadic}
For every $1<p<\infty$ and every integer $N\geq0$,
\[
 D_{p,N}\leq D_{p,N+1}\leq p'.
\]
Moreover,
\[
 \lim_{N\to\infty}D_{p,N}=p'.
\]
In particular, for every $\tau>0$ there are an integer $N\geq1$ and a
nonzero $F:\Omega_N\to[0,\infty)$ such that
\[
 \left\|\max_{0\leq n\leq N}E_nF\right\|_{L^p(\Omega_N)}
 \geq(p'-\tau)\|F\|_{L^p(\Omega_N)}.
\]
\end{proposition}

Given its role in the sequel, we include a detailed proof of the above proposition for the reader's convenience.

\begin{proof}
Identify $\Omega_N$ with the atoms of the dyadic $\sigma$-algebra of depth $N$
on $[0,1)$.  Under this identification, $E_n$ is conditional expectation
onto the dyadic $\sigma$-algebra of depth $n$.  Doob's inequality gives
$D_{p,N}\leq p'$.  If a function on $\Omega_N$ is regarded as a function
on $\Omega_{N+1}$ independent of the last coordinate, all its conditional
expectations are unchanged, and the new terminal expectation equals the
function itself, which was already the $n=N$ term.  Hence
$D_{p,N}\leq D_{p,N+1}$.

It remains to identify the limit. For $x\in[0,1)$, let
$\mathcal DF(x)$ be the supremum of
\[
 \left|\frac1{|I|}\int_I F(y)\,\dd y\right|
\]
over all dyadic intervals $I\subset[0,1)$ containing $x$.
Its sharp $L^p$ norm is $p'$; see
\cite[(1.3) and the paragraph following it, p.~311]{Melas2005}.  If $F$ is
a dyadic step function of depth $N$, then all conditional expectations
beyond depth $N$ equal $F$, and therefore
\[
 \|\mathcal DF\|_p\leq
 \left(\sup_{M\geq1}D_{p,M}\right)\|F\|_p.
\]
Finite-depth dyadic step functions are dense among the nonnegative
functions in $L^p([0,1))$.  Moreover,
\[
 |\mathcal DF-\mathcal DG|\leq\mathcal D(F-G)
\]
for nonnegative $F,G$, so Doob's inequality allows passage to the limit.
It follows that $\|\mathcal D\|_{L^p\to L^p}\leq\sup_M D_{p,M}$.
Sharpness and the already proved upper bound give
$p'\leq\sup_M D_{p,M}\leq p'$.  Hence monotonicity yields
$D_{p,N}\uparrow p'$.  Given $\tau>0$, choose $N$ so that
$D_{p,N}>p'-\tau$; the final statement then follows from
\eqref{eq:finite-dyadic-definition}.
\end{proof}

\section{Proof of Theorem \ref{thm:local-compactness}.}
\label{sec:local-compactness}

In this section, we present the proof of Theorem \ref{thm:local-compactness}. Namely, we show that compactness of one localized semigroup operator suffices for the associated maximal operator to have norm $p'$.

\medskip

\textit{Proof of Theorem \eqref{thm:local-compactness}.}
By Theorem~\ref{thm:stein-maximal}, only the lower bound remains. Without loss of generality, $0<\mu(E)<\infty$.
Indeed, by $\sigma$-finiteness there is a measurable
$\widetilde E\subseteq E$ with $0<\mu(\widetilde E)<\infty$.
The map $f\mapsto T_{t_0}(\one_{\widetilde E}f)$ is compact,
being the composition of the compact map
$f\mapsto T_{t_0}(\one_Ef)$ with the bounded map
$f\mapsto\one_{\widetilde E}f$.

Replacing $\mu$ by $\mu/\mu(E)$, we may assume $\mu(E)=1$.
Fix $1<p<\infty$, an integer $N\ge1$, and $\varepsilon>0$.
Recall the definitions of $\Omega_N$, $E_n$, and $D_{p,N}$ from
Section~\ref{sec:prelim}. In what follows, we shall construct an isometric lattice embedding $J: L^p(\Omega_N) \to L^p(X)$ and numbers $s_0> \ldots>s_N>0$ such that 
\[
  \|T_{s_n}JF-JE_nF\|_{L^p(\mu)} \le 2^N\varepsilon \|F\|_{L^p(\Omega_N)}
\]
holds for all $F \in L^p(\Omega_N)$, $0 \le n \le N$. Since $(E,\mu|_E)$ is a nonatomic probability space, we can choose independent
Rademacher functions $(r_m)_{m\ge1}$ on it and extend them by $1$
outside $E$.
By \eqref{eq:Lp-strong-continuity}, there is $0<s_0<t_0$ such that
\[
 \|T_{s_0}\one_E-\one_E\|_p<\varepsilon.
\]

We choose $k_1<\cdots<k_N$ and $s_0>\cdots>s_N>0$ inductively.
Put $k_0=0$ and, for the indices already chosen, write
\[
 w_S=\one_E\prod_{j\in S}r_{k_j},
 \qquad w_\varnothing=\one_E.
\]
Suppose that $k_1,\ldots,k_{j-1}$ and $s_0,\ldots,s_{j-1}$ are fixed.
Fix $S\subseteq\{1,\ldots,j-1\}$ and set $v_m=w_Sr_m$
for $m>k_{j-1}$.
Since
\[
 \langle v_m,v_\ell\rangle_{L^2(\mu)}
 =\int_Er_mr_\ell\,\dd\mu=\one_{\{m=\ell\}},
\]
by Bessel's inequality, we have $v_m\rightharpoonup0$ weakly in $L^2(\mu)$. Moreover,
$v_m=\one_Ev_m$, so the compactness assumption  gives
\[
 \|T_{t_0}v_m\|_2\longrightarrow0.
\]
On $L^2(\mu)$, write $T_t=e^{-tA}$, where $A$ is nonnegative and
self-adjoint, and let $\mathcal P$ be the spectral measure corresponding to $A$. The spectral theorem and H\"older's inequality give,
for each fixed $0<s<t_0$,
\[
 \|T_sv_m\|_2=\bigg(\int_0^{\infty}e^{-sa}\mathcal \, d\mathcal{P}_{v_m,v_m}(a)\bigg)^{1/2}
 \le\|T_{t_0}v_m\|_2^{s/t_0}\|v_m\|_2^{1-s/t_0}
 =\|T_{t_0}v_m\|_2^{s/t_0}\longrightarrow0.
\]
Contractivity gives $\|T_sv_m\|_1\le1$ and $\|T_sv_m\|_\infty\le1$.
Interpolating with the $L^2$ estimate, we obtain
\[
 \|T_sv_m\|_p
 \le\|T_sv_m\|_2^{\min\{2/p,\,2/p'\}}
 \longrightarrow0.
\]

The preceding limits hold for each of the finitely many pairs $(i,S)$
with $0\le i<j$ and $S\subseteq\{1,\ldots,j-1\}$. Choose $k_j>k_{j-1}$
so that
\begin{equation}\label{eq:local-compactness-annihilation}
 \|T_{s_i}w_{S\cup\{j\}}\|_p<\varepsilon
\end{equation}
 holds simultaneously for all such pairs. Having fixed $k_j$, use
\eqref{eq:Lp-strong-continuity} to choose $0<s_j<s_{j-1}$ such that
\begin{equation}\label{eq:local-compactness-preservation}
 \|T_{s_j}w_S-w_S\|_p<\varepsilon
 \qquad\bigl(S\subseteq\{1,\ldots,j\}\bigr).
\end{equation}

 Once the induction is complete, define
\[
 JF=\one_EF(r_{k_1},\ldots,r_{k_N}),
 \qquad F:\Omega_N\to\mathbb C.
\]
Independence gives
\[
 \|JF\|_{L^p(\mu)}^p
 =2^{-N}\sum_{\omega\in\Omega_N}|F(\omega)|^p
 =\|F\|_{L^p(\Omega_N)}^p.
\]
Also $J|F|=|JF|$, so $J$ is an isometric lattice embedding.
 
 For $S\subseteq \lbrace 1,...,N\rbrace$, we define the Walsh characters on $\Omega_N$ as
$$W_S(\omega)=\prod_{j\in S}\omega_j~,$$
and observe that the conditional expectation acts diagonally on them, namely
\begin{equation}
 E_nW_S=
 \begin{cases}
  W_S,&S\subseteq\{1,\ldots,n\},\\
  0,&S\nsubseteq\{1,\ldots,n\}.
 \end{cases}
 \label{eq:Walsh-conditional}
\end{equation}
We have $JW_S=w_S$. If $S\subseteq\{1,\ldots,n\}$, apply
\eqref{eq:local-compactness-preservation} with $j=n$; the initial choice
of $s_0$ covers $n=0$. Otherwise, let $j=\max S>n$ and apply
\eqref{eq:local-compactness-annihilation} to $S\setminus\{j\}$ at time
$s_n$. Together with \eqref{eq:Walsh-conditional}, these estimates give
\[
 \|T_{s_n}JW_S-JE_nW_S\|_p\le\varepsilon
 \qquad(0\le n\le N).
\]

Expand $F=\sum_{S\subseteq\{1,\ldots,N\}}c_SW_S$. Since
$|c_S|\le\|F\|_{L^1(\Omega_N)}\le\|F\|_{L^p(\Omega_N)}$,
\begin{align*}
 \left\|\max_{0\le n\le N}|T_{s_n}JF-JE_nF|\right\|_p
 &\le\sum_{n=0}^N\sum_S|c_S|
       \|T_{s_n}JW_S-JE_nW_S\|_p\\
 &\le(N+1)2^N\varepsilon\|F\|_{L^p(\Omega_N)}.
\end{align*}
For $F\ge0$, the positivity of and preservation of finite maxima of $J$ yields
\begin{align*}
 \|T^*JF\|_p
 &\ge\left\|\max_{0\le n\le N}T_{s_n}JF\right\|_p\\
 &\ge\left\|\max_{0\le n\le N}E_nF\right\|_{L^p(\Omega_N)}
       -(N+1)2^N\varepsilon\|F\|_{L^p(\Omega_N)}.
\end{align*}
Divide by $\|JF\|_p=\|F\|_{L^p(\Omega_N)}$ for nonzero $F\ge0$.
Taking the supremum in \eqref{eq:finite-dyadic-definition} and then
letting $\varepsilon\downarrow0$ with $N$ fixed gives
\[
 D_{p,N}\le\|T^*\|_{p\to p}\le p'.
\]
Let $N\to\infty$ and apply Proposition~\ref{prop:finite-dyadic}.
This proves the theorem.

\qed

\section{Applications to norms of maximal functions} 
\label{sec: applications}

\subsection{Maximal functions of semigroups}
Without attempting to be exhaustive, in this subsection we present several examples of classical semigroups that fit within the framework of Theorem \ref{thm:local-compactness}. For definitions of the semigroups in items (c)--(h), see the references cited below. 
\begin{corollary} 
\label{cor: semigroup examples}
Let $d\geq 1$, $1<p<\infty$.
Let $(T_t)_{t\ge0}$ be one of the following strongly continuous positive symmetric sub-Markovian semigroups: 
\begin{enumerate}[label=(\alph*)]
    \item The heat semigroup on $L^p(\R^d,\dd x)$,  $T_t=\mathcal{H}_t^d=e^{t \Delta}$. 
    \item The periodic heat semigroup on $L^p(\mathbb{T}^d,\dd m)$ with $\dd m$ the normalized Lebesgue measure on $\mathbb{T}^d$. 
    \item The Dunkl-semigroup $T_t=e^{t\Delta_{\kappa}}$ associated to a general root system and a non-negative multiplicity function $\kappa$; see \cite{DziubanskiHejna} or \cite{Rosler}. 
       \item The Schr\"odinger semigroup on $L^p(\R^d,\dd x)$, $T_t= e^{-tH}$ with $H=-\Delta + V,$ where $V\ge0$ and $V\in L^1_{{\rm loc}}$; see \cite[Section 2]{KucharskiWrobel}.

\item The heat semigroup $T_t=e^{t\Delta_M}$ on
$L^p(M,\mu)$, where $M$ is a complete connected
$d$-dimensional Riemannian manifold without boundary,
$\mu$ is its Riemannian volume measure, and $\Delta_M$
is the nonpositive self-adjoint Laplace--Beltrami operator;
see \cite[Sections~2.2 and~3.1--3.2]{Grigoryan2006}.
       
       \item The Ornstein-Uhlenbeck semigroup on $L^p(\R^d,\gamma_d)$, $T_t=e^{tL}$ with $L=\Delta - x\cdot\nabla$, where $\gamma_d$ is the standard Gaussian measure; see \cite[Section~2.2, (2.27), Theorem 2.5]{Urbina2019} (after a change of Gaussian normalization).

\item The heat semigroup associated with Laguerre polynomials on $L^p((0,\infty)^d,\dd \mu_\alpha) $ with
\[
 \dd\mu_\alpha(x)=\prod_{j=1}^d\frac{x_j^{\alpha_j}e^{-x_j}}
                         {\Gamma(\alpha_j+1)}\,\dd x,
\]
where $\alpha_j>-1$ for $1\leq j\leq d$;
see \cite[Section~2]{NowakStempakContractivity}.

                             \item  The heat semigroup associated with Jacobi polynomials on $L^p((-1,1)^d, \dd\mu_{\alpha,\beta})$ 
with
\[
 \dd\mu_{\alpha,\beta}(x)=\prod_{j=1}^d
   \frac{\Gamma(\alpha_j+\beta_j+2)
         (1-x_j)^{\alpha_j}(1+x_j)^{\beta_j}}
        {2^{\alpha_j+\beta_j+1}
         \Gamma(\alpha_j+1)\Gamma(\beta_j+1)}\,\dd x,
\]
where $\alpha_j,\beta_j>-1$ for $1\leq j\leq d$;
see \cite[Sections~2 and~5]{NowakSjogren2013}.

\end{enumerate}
Then we have $$\|T^\ast\|_{L^p(\mu)\rightarrow L^p(\mu)}= p'~.$$
\end{corollary}
\proof
 Note first that all the described semigroups are sub-Markovian. This is either explicitly stated in the provided references or can be easily deduced from them. 

The common idea in the proofs of items (a)-(e) is to show that $T_t(\one_E)$ is a Hilbert-Schmidt operator (hence compact) for a suitable set $E$ and any $t>0.$ This will be deduced from the fact that the kernel of $f\mapsto T(\one_E f)$ belongs to the $L^2$ space over the product domain.

We start with the proofs of items (a)--(c). Taking the multiplicity
function $\kappa=0$, we see that (a) is a subcase of (c).
Furthermore, once (a) is established, (b) follows from a maximal
transference principle of Kenig and Tomas
\cite[Theorem~1]{KenigTomas}. We note that, while the preservation
of the constants is not explicitly stated in
\cite[Theorem~1]{KenigTomas}, it can be read from the proof;
see also the first paragraph in \cite[p.~167]{AuscherCarro}. Thus it remains to establish c), for which we shall use the notation in \cite{DziubanskiHejna}.  Denote by $h_t(x,y)$ the Dunkl heat kernel, let $w$ be the Dunkl measure and take $E$ to be a Euclidean ball in $\R^d$. Then, using symmetry, the semigroup property of $h_t$,
and the estimates \cite[eq.~(1.13), eq.~(2.1)]{DziubanskiHejna},
we obtain
\begin{equation*}
\int_E \int_{\R^d} h_t(x,y)^2\,\mathrm dw(x)\,\mathrm dw(y)=\int_{E}h_{2t}(y,y)\,\mathrm dw(y) <\infty,\qquad t>0.
\end{equation*}
This implies that $T_t(\one_E)$ is a Hilbert-Schmidt operator; hence, it is compact. An application of Theorem \ref{thm:local-compactness} completes the proof for items a)-c).

Now, item d) follows from a). Indeed, using the Feynman--Kac formula \cite{Sznitman}
(or the Trotter product formula), we see that, for any $f\ge0$,
\[
e^{-tH}f(x)\le \mathcal H_t^d f(x), \qquad x-a.e.\
\]
Thus, repeating the argument used to prove (c) with $\kappa=0$
and $E$ a Euclidean ball, we conclude that the operator
$f\mapsto e^{-tH}(\one_E f)$ is Hilbert--Schmidt.
In particular, it is compact, and
Theorem~\ref{thm:local-compactness} applies.

For (e), the required semigroup properties and smoothness of
its heat kernel $h_t$ follow from
\cite[Theorems~3.1 and~3.3, and (3.11)--(3.12)]{Grigoryan2006}.
For a nonempty relatively compact open set $E\subset M$,
symmetry and the semigroup identity give, for every $t>0$,
\[
\begin{aligned}
 \int_E\int_M h_t(x,y)^2\,\dd\mu(x)\,\dd\mu(y)
 =\int_E h_{2t}(y,y)\,\dd\mu(y)
 \leq\mu(E)\sup_{y\in\overline E}h_{2t}(y,y)<\infty.
\end{aligned}
\]
Thus $f\mapsto T_t(\one_Ef)$ is Hilbert--Schmidt, hence
compact. Note that Riemannian volume is nonatomic and $\sigma$-finite,
so Theorem~\ref{thm:local-compactness} applies.

Each of the semigroups in items (f)--(h) is associated to an orthogonal expansion. In particular, their negative generators have discrete spectrum on $L^2$, their eigenvalues can be ordered in an increasing sequence $\lambda_n\to \infty,$ and the corresponding eigenspaces are finite dimensional. Therefore $T_t$ itself is compact for any $t>0$ and Theorem \ref{thm:local-compactness} applies with $E$ being the full space.

\qed

\par

Let $T_t=e^{-tA}$ satisfy the hypotheses of
Theorem~\ref{thm:local-compactness}, where $-A$ is its $L^2$ generator and $A$ is nonnegative and self-adjoint. The next corollary establishes a general principle: the conclusion of Theorem \ref{thm:local-compactness} is inherited by semigroups subordinate to $T_t$. To make this precise, for
a Bernstein function $\Phi$ with $\Phi(0)=0$, set
\[
 S_t=e^{-t\Phi(A)},
 \qquad
 S^*f=\operatorname*{ess\,sup}_{t>0}|S_tf|.
\]
\begin{corollary}\label{cor:bernstein-maximal-norms}
Suppose that $(T_t)_{t\geq0}$ satisfies the hypotheses of Theorem~\ref{thm:local-compactness}.
If $\Phi(\lambda)\to\infty$ as $\lambda\to\infty$, then
\[
 \|S^*\|_{L^p(\mu)\to L^p(\mu)}=p',
 \qquad 1<p<\infty.
\]
\end{corollary}

\begin{proof}
By Bochner subordination, $(S_t)_{t\geq0}$ is a strongly continuous
positive symmetric sub-Markovian semigroup; see \cite[Theorem~5.2, Proposition~13.1, and Remark~13.4]{SSV2012}.
Let $t_0,E$ be as in Theorem~\ref{thm:local-compactness}. For $t,R>0$, the operator
\[
 f\longmapsto
 S_t\mathbf1_{[0,R]}(A)(\mathbf1_Ef)
 =\bigl(S_te^{t_0A}\mathbf1_{[0,R]}(A)\bigr)
   T_{t_0}(\mathbf1_Ef)
\]
is compact, since the first factor is bounded on $L^2(\mu)$ and the
map $f\mapsto T_{t_0}(\mathbf1_Ef)$ is compact. Moreover,
\[
 \bigl\|S_t(\mathbf1_Ef)
       -S_t\mathbf1_{[0,R]}(A)(\mathbf1_Ef)\bigr\|_2
 \leq e^{-t\Phi(R)}\|f\|_2,
 \qquad f\in L^2(\mu).
\]
Since $\Phi(R)\to\infty$, these compact operators converge in operator
norm to $f\mapsto S_t(\mathbf1_Ef),$ which is therefore also compact. Theorem~\ref{thm:local-compactness} gives the conclusion.
\end{proof}
In particular, the maximal operators of the Poisson semigroups
on $\R^d$, $\mathbb T^d$, and the manifolds in item (h) of
Corollary~\ref{cor: semigroup examples} all have $L^p$ norm $p'$;
in the manifold case the semigroup is $e^{-t\sqrt{-\Delta_M}}$.

The next corollary shows that the conclusions of Theorem \ref{thm:local-compactness} can be extended slightly beyond the localized compactness hypothesis. 
\begin{corollary}
 Let $(S_t)_{t\geq 0}$ be a strongly continuous positive symmetric sub-Markovian semigroup on a $\sigma$-finite measure space $(X,\mathcal F,\mu)$. Suppose that there exist a semigroup $(T_t)_{t\geq 0}$, acting on a measure space $(Y,\mathcal B,\nu)$ and satisfying the hypotheses of Theorem \ref{thm:local-compactness}, and a positive linear embedding
\[
J:L^0(Y)\to L^0(X)
\]
such that, for every $1<p<\infty$, $J$ restricts to an isometric embedding from $L^p(Y)$ into $L^p(X)$, and
\[
S_tJ=JT_t,\qquad t>0.
\]
Then, for every $1<p<\infty$,
\[
\|S^*\|_{L^p(X,\mu)\to L^p(X,\mu)}=p'.
\]
\end{corollary} 
\begin{proof}  By Theorem \ref{thm:stein-maximal}, only the lower bound remains to be proven.
Since \(J\) is a positive \(L^p\)-isometry, it is a lattice embedding. Note that, for every $f\in L^p(Y)$ with $f\ge0$ and all
$t_1,\ldots,t_N>0$, by positivity and preservation of finite maxima of $J$, we have
\[
J\Big( \sup_{1 \le i \le N} T_{t_i} f \Big)= \sup_{1 \le i \le N} JT_{t_i} f= \sup_{1 \le i \le N} S_{t_i} Jf. 
\]
Therefore, since $J$ is an $L^p$-isometry, we get
\[
\Big\| \sup_{1 \le i \le N} T_{t_i} f \Big\|_{L^p(Y)} \le \|S^*(Jf) \|_{L^p(X)} \le  \|S^* \|_{L^p(X) \to L^p(X)} \|f \|_{L^p(Y)}. 
\]
Since $N$ and $t_1,\dots,t_N$ are arbitrary, this yields
\[
\|T^*f\|_{L^p(Y)} \le \|S^*\|_{L^p(X) \to L^(X)} \|f \|_{L^p(Y)}.
\]
Using positivity and the arbitrariness of $f\ge0$, we obtain
\[
p'=\|T^*\|_{L^p(Y) \to L^p(Y)} \le \|S^* \|_{L^p(X) \to L^p(X)},
\]
which finishes the proof.
\end{proof}

 The previous corollary applies in particular to infinite-dimensional semigroups arising as infinite tensor products of Markov semigroups acting on probability spaces and satisfying the hypotheses of Theorem \ref{thm:local-compactness}.
\begin{corollary}
 For each $i \ge 1$, let $(T_t^{(i)})_{t \ge 0}$ be a Markov semigroup satisfying the hypotheses of Theorem \ref{thm:local-compactness} and
acting on a probability space $(X_i, \mathcal{F}_i, \mu_i)$. Set
\[
X=\prod_{i=1}^\infty X_i, \quad \mu=\bigotimes_{i=1}^\infty \mu_i, \quad S_t= \bigotimes_{i=1}^\infty T_t^{(i)}.
\]
Then, for every $1<p<\infty$, 
\[
\|S^*\|_{L^p(X,\mu) \to L^p(X,\mu)}=p'.
\]
\end{corollary}

\begin{remark}
Let $1<p<\infty$, and let $(S_t)_{t\ge0}$ be the
infinite-dimensional Ornstein--Uhlenbeck semigroup on
$L^p(\R^{\mathbb{N}},\gamma^{\otimes\mathbb{N}})$,
where $\gamma$ is the standard Gaussian measure on $\R$.

This semigroup is given by the Mehler formula
\[
S_t f(x)
=
\int_{\R^{\mathbb{N}}}
f\bigl(e^{-t}x+\sqrt{1-e^{-2t}}\,y\bigr)
\,d\gamma^{\otimes\mathbb{N}}(y),
\qquad t>0.
\]
For the underlying Gaussian product space and the basic properties
of this semigroup, see \cite{Bogachev1998}.
As before, set
\[
S^\ast f=\esssup_{t>0}S_t|f|.
\]
Since the semigroup $(S_t)_{t\ge0}$ is the infinite tensor product
of the one-dimensional Ornstein--Uhlenbeck semigroups, the preceding corollary yields
    $$\|S^\ast\|_{L^p(X)\rightarrow L^p(X)}=p'~.$$
\end{remark}

\subsection{Hardy-Littlewood maximal functions}

 In this subsection we shall apply Theorem \ref{thm:local-compactness} to deduce Corollaries \ref{cor:HL-norm-R} and \ref{cor:ball-lower-bounds}.
In dimension one, the upper bound follows from the corresponding bounds for the two one-sided maximal operators.  For
\(F\in L^p(\mathbb R)\), define
\[
 M^+F(x)=\sup_{r>0}\frac1r\int_0^r|F(x+y)|\,\dd y,
 \qquad
 M^-F(x)=\sup_{r>0}\frac1r\int_0^r|F(x-y)|\,\dd y.
\]
Let us recall a classical bound for $L^p$ norms of the operators defined above.
\begin{lemma}[{\cite[Theorem 21.76(i)]{HewittStromberg}}]
\label{lem:one-sided-sharp}
For every $1<p<\infty$, every $F\in L^p(\mathbb R)$, and either choice of sign, we have
\[
 \|M^\pm F\|_{L^p(\mathbb R)}\leq p'\|F\|_{L^p(\mathbb R)}.
\]
\end{lemma}
Recall the definition
\[
\mathcal{H}_t^dF=e^{t \Delta}F=F \ast h_t^{d}, \quad \text{where } \ h_t^{d}(x)= \frac{1}{(4 \pi t)^{d/2}} e^{-|x|^2/(4t)}.
\]
Using the standard domination of the heat semigroup maximal operator by the Hardy–Littlewood maximal operator, we can conclude the proof of Corollaries \ref{cor:ball-lower-bounds} and \ref{cor:HL-norm-R}.
\begin{proof}[Proof of Corollaries \ref{cor:ball-lower-bounds} and \ref{cor:HL-norm-R}]
Let $\omega_d=|B_1^d|$ and
\[
\varphi_t(r)= \frac{1}{(4 \pi t)^{d/2}} e^{-r^2/(4t)}.
\]
Note that by polar coordinates and integration by parts we have
\begin{equation} \label{eq: varphi normalization to 1}
    1= \int_{\R^d} h_t(x) dx= d \omega_d \int_0^\infty \varphi_t(r) r^{d-1} dr= - \int_0^\infty \varphi_t'(r) |B_r^d| dr.
\end{equation}
Let $\sigma$ be surface measure on $\mathbb S^{d-1}$, and write
\[
 \mathcal B_u^dF(x)=\frac{1}{|B_u^d|}\int_{B_u^d}F(x-y)\,\dd y.
\]
Using polar coordinates again, we obtain, for any $F\in L^p(\R^d)$, 
\begin{align*}
    \mathcal{H}_t^d F(x)&= \int_{\mathbb S^{d-1}}\int_0^\infty F(x-r \omega) \varphi_t(r) r^{d-1} \dd r \dd \sigma(\omega) \\
   &= - \int_{0}^\infty \varphi_t'(u)\int_{\mathbb S^{d-1}}\int_{0}^u F(x-r \omega) r^{d-1} \dd r \dd\sigma(\omega)\,\dd u\\
   &= - \int_{0}^\infty \varphi_t'(u) |B_u^d| \cdot \mathcal{B}_u^d F(x) \dd u.
\end{align*}
The above identity combined with \eqref{eq: varphi normalization to 1} and H\"older's inequality gives us 
\[
\|\sup_{t>0} |\mathcal{H}_t^dF| \|_{L^p(\R^d)} \le \|\sup_{t>0} |\mathcal{B}_t^dF| \|_{L^p(\R^d)}.
\]
Since $F \in L^p(\R^d)$ is arbitrary, point (a) of Corollary \ref{cor: semigroup examples} combined with the inequality above concludes the proof of Corollary \ref{cor:ball-lower-bounds}.
For the proof of Corollary \ref{cor:HL-norm-R} it suffices to establish upper bound. Note that we have
\[
 \mathscr B_*^1F\leq\frac12(M^-F+M^+F),
\]
hence Lemma~\ref{lem:one-sided-sharp} gives us
\[
 \|\mathscr B_*^1\|_{L^p(\mathbb R)\to L^p(\mathbb R)}\leq p'
\]
and finishes the proof.
\end{proof}

\appendix

\section{Strict lower bound for the Hardy-Littlewood maximal operator}\label{sec:strict-HL-lower-bound}
We conclude by proving Proposition \ref{prop:HL-norm-Rd}. Before turning to the proof, let us make a few informal comments. The lattice $\Lambda$
used below is chosen to have covolume one and to satisfy $|B_r|>2$
for some $r<\operatorname{dist}(0,\Lambda\setminus\{0\})$.
For $2\le d\le10$, one can simply take $\Lambda=\mathbb Z^d$.
A one-dimensional version of the argument also gives a lower bound
arbitrarily close to $p'$ for the centered Hardy--Littlewood maximal
operator.

The construction is very closely related to a stopped martingale on
$(\{0,1\}^n,\mu^{\otimes n})$, where $\mu(\{0\})=\delta$ and
$\mu(\{1\})=1-\delta$. Starting from $1$, it multiplies by $a$ at
each $1$, and by $b$ at the first $0$, after which it stops.
With the parameters chosen below, this martingale nearly attains
equality in Doob's $L^p$ inequality as first $n\to\infty$ and then
$\delta\to0$. Taking periodic extensions, with dilation by $L$
at each step, allows us to mimic this construction in $\mathbb R^d$,
but with $\delta$ fixed. The condition $|B_r|>2$ gives us an
additional gain, allowing us to exceed $p'$. 
\begin{proof}[Proof of Proposition~\ref{prop:HL-norm-Rd}]
We construct a nonnegative $f\in C_c^\infty(\mathbb R^d)$ such that
$\|\mathscr B_*^d f\|_p>p'\|f\|_p$. Fix $d\ge2$ and $1<p<\infty$. By the Minkowski--Hlawka theorem (see \cite{Cassels} or \cite[Satz~4, p.~298]{Hlawka}), there are a lattice $\Lambda$ of
covolume one and $r>0$ such that
\[
 |B_r|>2,
 \qquad \dist(0,\Lambda\setminus\{0\})>r.
\]
Let $Q$ be a half-open fundamental parallelepiped of $\Lambda$ with
a vertex at the origin, so that $|Q|=1$. Choose an integer $k\ge2$
such that $E=k^{-1}Q$ satisfies
\[
 \operatorname{diam}(E)
 <\min\{r,\dist(0,\Lambda\setminus\{0\})-r\}.
\]
The triangle inequality gives
\[
 (x+B_r)\cap(E+\Lambda)=E\qquad(x\in E),
\]
where $E+\Lambda=\bigcup_{\lambda\in\Lambda}(E+\lambda)$.
Throughout the proof, functions on $Q$ are extended
$\Lambda$-periodically.

Put $\delta=|E|=k^{-d}$ and set
\[
 a=(1-\delta)^{-1/p},
 \qquad
 b=\frac{1-(1-\delta)a}{\delta}.
\]
Then $0<b<1<a$ and
\[
 (1-\delta)a+\delta b=1,
 \qquad (1-\delta)a^p=1.
\]
Fix an integer $n\ge1$ and $L$, which is a positive
multiple of $k$. Define on $Q$
\[
 f_0=1,
 \qquad
 f_{j+1}(x)=b\one_E(x)
       +a\one_{Q\setminus E}(x)f_j(Lx),
 \qquad 0\le j<n.
\]

In particular,
\[
 f_1=b\one_E+a\one_{Q\setminus E}.
\]
\begin{figure}[!htbp]
\centering
\resizebox{\linewidth}{!}{%
\begin{tikzpicture}[
    x=1cm,y=1cm,font=\small,
    frame/.style={densely dashed,draw=black!65,line width=.45pt},
    lead/.style={-{Stealth[length=1.15mm]},line width=.4pt},
    heading/.style={anchor=west,font=\small\bfseries},
    patch/.style={fill=Bfill,draw=Bedge,line width=.2pt},
    finepatch/.style={fill=ABfill,draw=ABedge,line width=.15pt}
]
\path[use as bounding box] (0,0) rectangle (15.25,5.7);
\node[heading] at (.05,5.4) {(a)\quad First step: $f_1$};
\node[heading] at (5.18,5.4) {(b)\quad Second step: $f_2$};
\node[heading] at (10.33,5.4) {(c)\quad Enlarged neighborhood};

\begin{scope}[shift={(2.20,2.58)},x=1.75cm,y=1.75cm]
  \foreach \i in {-1,0,1}{
    \foreach \j in {-1,0,1}{
      \path[patch] (\i,\j) rectangle ++(.1,.1);
    }
  }
  \draw[frame] (0,0) rectangle (1,1);
  \node at (.85,.87) {$Q$};
  \draw[line width=.65pt] (.05,.05) circle[radius=.85];
  \draw[line width=.4pt] (.05,.05)--(.65104,.65104)
     node[pos=.67,above left,inner sep=1pt] {$r$};
  \fill (.05,.05) circle[radius=.012];
  \node[anchor=north west,inner sep=1pt] at (.10,.02) {$x$};
  \node at (.39,-.45) {$a$};
  \node[anchor=east,inner sep=1pt,text=Bedge] at (-.08,.35) {$b$ on $E$};
  \draw[lead,draw=Bedge] (-.08,.29)--(.015,.09);
\end{scope}
\node[align=center] at (2.28,.25)
  {$x\in E$};

\begin{scope}[shift={(5.18,.77)},x=4.12cm,y=4.12cm]
  \path[patch] (0,0) rectangle (.1,.1);
  \foreach \i in {0,...,19}{
    \foreach \j in {0,...,19}{
      \pgfmathtruncatemacro{\outside}{(\i>=2)||(\j>=2)}
      \ifnum\outside=1
        \path[finepatch] ({\i/20},{\j/20}) rectangle ++(.005,.005);
      \fi
    }
  }
  \draw[line width=.5pt] (0,0) rectangle (1,1);
  \node[anchor=south east,inner sep=2pt] at (1,1) {$Q$};
  \draw[line width=.45pt] (.5025,.5025) circle[radius=.0425];
  \draw[draw=black!70,line width=.4pt] (.44,.44) rectangle (.5625,.5625);
  \coordinate (zoomcorner) at (.5625,.5625);
  \coordinate (originalE) at (.05,.05);
\end{scope}
\node[anchor=north,inner sep=1pt,text=Bedge] at (5.386,.73) {$E$};
\path[patch] (6.15,.13) rectangle ++(.14,.14);
\node[anchor=west,inner sep=2pt,text=Bedge] at (6.31,.20) {$b$ on $E$};
\path[finepatch] (7.73,.13) rectangle ++(.07,.07);
\node[anchor=west,inner sep=2pt,text=ABedge] at (7.84,.20) {$ab$};
\draw[line width=.3pt] (8.62,.13) rectangle ++(.14,.14);
\node[anchor=west,inner sep=2pt] at (8.80,.20) {$a^2$};

\begin{scope}[shift={(12.55,2.58)},x=1.75cm,y=1.75cm]
  \foreach \i in {-1,0,1}{
    \foreach \j in {-1,0,1}{
      \path[finepatch] (\i,\j) rectangle ++(.1,.1);
    }
  }
  \draw[frame] (0,0) rectangle (1,1);
  \draw[line width=.65pt] (.05,.05) circle[radius=.85];
  \draw[line width=.4pt] (.05,.05)--(.65104,.65104)
     node[pos=.67,above left,inner sep=1pt] {$r/L$};
  \fill (.05,.05) circle[radius=.012];
  \node[anchor=north west,inner sep=1pt] at (.10,.02) {$y$};
  \node at (.39,-.45) {$a^2$};
  \node[anchor=east,inner sep=1pt,text=ABedge] at (-.29,.24) {$ab$};
  \draw[lead,draw=ABedge] (-.28,.18)--(.015,.075);
  \coordinate (zoomleft) at (-1.20,1.25);
\end{scope}
\draw[lead,draw=black!60] (zoomcorner) -- (9.65,4.75) -- (zoomleft);
\node[align=center] at (12.59,.25)
  {$\operatorname{dist}(y,E+\Lambda)>r/L$};
\end{tikzpicture}%
}
\caption{Visualization of the first two steps of the construction for $\Lambda=\mathbb Z^2$.}
\label{fig:lattice-construction}
\end{figure}
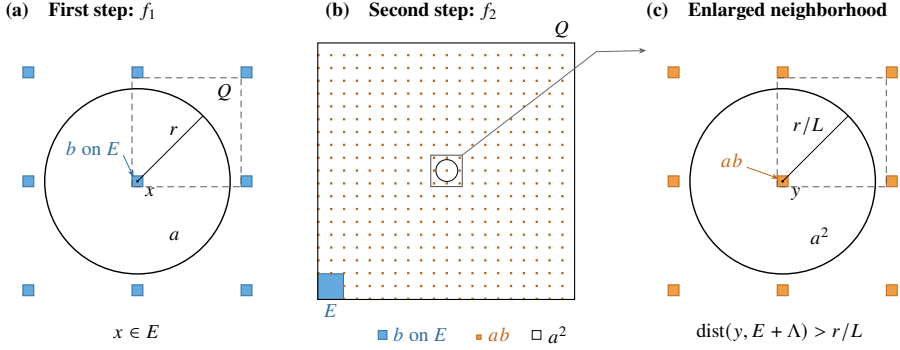
We suppress the dependence of $f_j$ on $L$. For every
$\Lambda$-periodic $h\in L^1(Q)$, a change of variables gives
\begin{equation}\label{eq:appendix-cell-identity}
 \int_{L^{-1}(Q+\lambda)}h(Lx)\,\dd x
 =L^{-d}\int_Qh(x)\,\dd x,
 \qquad \lambda\in\Lambda.
\end{equation}
Since $L$ is a multiple of $k$, both $E$ and $Q\setminus E$ are
unions of these cells. Induction therefore gives
\begin{align*}
 \int_Qf_{j+1}(x)\,\dd x
 &=\delta b+a(1-\delta)\int_Qf_j(x)\,\dd x=1,\\
 \int_Qf_{j+1}(x)^p\,\dd x
 &=\delta b^p+a^p(1-\delta)\int_Qf_j(x)^p\,\dd x\\
 &=\delta b^p+\int_Qf_j(x)^p\,\dd x.
\end{align*}
Consequently,
\begin{equation}\label{eq:appendix-moments}
 \int_Qf_j(x)^p\,\dd x=1+j\delta b^p,
 \qquad 0\le f_j\le a^n,
 \qquad 0\le j\le n.
\end{equation}

For $x\in E$, the choice of $E$ and the identity
$\delta(a-b)=a-1$ give
\[
 \frac{1}{|B_r|}\int_{x+B_r}f_1(y)\,\dd y
 =a-\frac{\delta(a-b)}{|B_r|}
 =a-\frac{a-1}{|B_r|}=:m.
\]
Since $|B_r|>2$, we have $m>(a+1)/2$. Strict convexity of
$t\mapsto t^{-1/p}$ gives
\[
 p'b
 =\frac1\delta\int_{1-\delta}^{1}t^{-1/p}\,\dd t
 <\frac{a+1}{2}<m.
\]

To pass from periodic functions to compactly supported functions,
we first restrict the averaging radii. For nonnegative
$h\in L^1_{\mathrm{loc}}(\R^d)$, write
\[
 \mathscr B_{*,r}^dh(x)
 =\sup_{0<t\le r}\frac{1}{|B_t|}
       \int_{x+B_t}h(y)\,\dd y.
\]
Then $\mathscr B_{*,r}^df_j$ is $\Lambda$-periodic and
$0\le\mathscr B_{*,r}^df_j\le a^n$. In what follows, $C_n$ may
depend on $n$ and on the fixed parameters, but not on $L$, $x$, or
$j\in\{0,\ldots,n-1\}$. Its value may change from line to line.

On each cell $L^{-1}(Q+\lambda)$, the function $f_{j+1}-f_1$
equals either zero or $a(f_j(L\,\cdot)-1)$, and hence has integral
zero. Thus, in its integral over $x+B_r$, only cells whose closures
meet $\partial(x+B_r)$ contribute. Their union is contained in
\[
 \left\{y\in\R^d:\bigl||y-x|-r\bigr|
       \le\frac{\operatorname{diam}(Q)}L\right\},
\]
whose measure is at most $C_n/L$, uniformly in $x$. Since
$|f_{j+1}-f_1|\le2a^n$, it follows that
\[
 \sup_{x\in E}
 \left|\frac{1}{|B_r|}\int_{x+B_r}f_{j+1}(y)\,\dd y-m\right|
 \le\frac{C_n}{L}.
\]
In particular,
\[
 \int_E(\mathscr B_{*,r}^df_{j+1})^p\,\dd x
 \ge\delta m^p-\frac{C_n}{L}.
\]
On the other hand, if $\dist(x,E+\Lambda)>r/L$, then, for
$0<t\le r$, we have (see Figure~\ref{fig:lattice-construction})
\[
 \frac{1}{|B_{t/L}|}\int_{x+B_{t/L}}f_{j+1}(y)\,\dd y
 =\frac{a}{|B_t|}\int_{Lx+B_t}f_j(y)\,\dd y.
\]
Taking the supremum in $t$ gives
\[
 \mathscr B_{*,r}^df_{j+1}(x)
 \ge a(\mathscr B_{*,r}^df_j)(Lx).
\]
Moreover,
\[
 \left|\left\{x\in Q\setminus E:
       \dist(x,E+\Lambda)\le r/L\right\}\right|
 \le\frac{C_n}{L},
\]
because this set lies within distance $r/L$ of finitely many faces
of translates of $E$. We estimate the contribution of this set using
$\mathscr B_{*,r}^df_j\le a^n$. Applying
\eqref{eq:appendix-cell-identity} to
$(\mathscr B_{*,r}^df_j)^p$ and using $a^p(1-\delta)=1$, we obtain
\begin{align*}
 \int_Q(\mathscr B_{*,r}^df_{j+1})^p\,\dd x
 &\ge\delta m^p
      +a^p\int_{Q\setminus E}(\mathscr B_{*,r}^df_j)(Lx)^p\,\dd x
      -\frac{C_n}{L}\\
 &=\delta m^p+\int_Q(\mathscr B_{*,r}^df_j)^p\,\dd x
      -\frac{C_n}{L}.
\end{align*}
Summing over $j=0,\ldots,n-1$ and using
$\mathscr B_{*,r}^df_0=1$ gives
\[
 \int_Q(\mathscr B_{*,r}^df_n)^p\,\dd x
 \ge1+n\delta m^p-\frac{C_n}{L}.
\]
Since $m>p'b$,
\[
 \lim_{n\to\infty}\frac{1+n\delta m^p}{1+n\delta b^p}
 =\left(\frac mb\right)^p>(p')^p.
\]
Together with \eqref{eq:appendix-moments}, this allows us to choose
first $n$ and then $L$ large enough so that
\[
 \|\mathscr B_{*,r}^df_n\|_{L^p(Q)}
 >p'\|f_n\|_{L^p(Q)}.
\]
Fix these choices of $n$ and $L$.

It remains to obtain a smooth compactly supported test function.
For a positive integer $N$, put $u=f_n\one_{NQ}$, where
$NQ=\{Nx:x\in Q\}$. Since $NQ$ is the disjoint union of $N^d$
translates of $Q$,
\[
 \|u\|_p^p=N^d\int_Qf_n(x)^p\,\dd x.
\]
If $x\in NQ$ and $\dist(x,\partial(NQ))>r$, then
$\mathscr B_{*,r}^du(x)=\mathscr B_{*,r}^df_n(x)$.
The remaining part of $NQ$ has measure at most $C_nN^{d-1}$.
Periodicity and $\mathscr B_{*,r}^df_n\le a^n$ therefore give
\[
 \|\mathscr B_{*,r}^du\|_p^p
 \ge N^d\int_Q(\mathscr B_{*,r}^df_n)^p\,\dd x
       -C_nN^{d-1}.
\]
The strict inequality on $Q$ now yields, for all sufficiently
large $N$,
\[
 \|\mathscr B_{*,r}^du\|_p>p'\|u\|_p.
\]
Fix such an $N$. By approximating $u$ arbitrarily well in
$L^p(\mathbb R^d)$ by smooth compactly supported functions and
using the boundedness of the maximal operator, we obtain the
same strict inequality for some smooth compactly supported function.
\end{proof}

\microtypesetup{expansion=false}

\end{document}